\documentclass[a4paper, 12pt, oneside]{article}
\pdfoutput=1 
\usepackage[T1]{fontenc}
\usepackage[british]{babel}
\usepackage{lmodern}
\usepackage{etoolbox}
\usepackage{amsmath, amsthm, amssymb, bm, mleftright}
\usepackage{booktabs, longtable, caption, subcaption, multirow}
\usepackage[space]{grffile}
\usepackage[margin=2.5cm, footskip=1cm]{geometry} 
\usepackage{enumitem}
\setenumerate[1]{label=(\arabic*), ref=\arabic*}
\usepackage[protrusion=allmath]{microtype}
\usepackage{tikz}
\usetikzlibrary{cd, calc}
\usepackage{listings}
\usepackage[colorlinks, citecolor=blue, linkcolor=blue, linktoc=section]{hyperref}
\usepackage{mathtools}
\usetikzlibrary{calc,decorations.pathmorphing,shapes}

\allowdisplaybreaks
\usepackage[capitalise, noabbrev]{cleveref}
\creflabelformat{enumi}{(#2#1#3)}
\creflabelformat{enumii}{(#2#1#3)}
\creflabelformat{enumiii}{(#2#1#3)}

\usepackage{titlesec}
\titleformat*{\section}{\large\bfseries}
\titleformat*{\subsection}{\normalsize\bfseries}
\newlength{\VerticalSpaceAfterParagraph}
\titlespacing*{\paragraph}{0pt}{\VerticalSpaceAfterParagraph}{1em}

\usepackage{titling}
\pretitle{\vspace{-\baselineskip}\begin{center}\Large\bfseries}
\posttitle{\end{center}\vspace{-0.25\baselineskip}}
\preauthor{\begin{center}}
\postauthor{\end{center}}
\predate{\begin{center}}
\postdate{\end{center}}

\setlist
  {
    topsep = 5.0pt plus 2.0pt minus 3.0pt,
    partopsep = 1.5pt plus 1.0pt minus 1.0pt,
    parsep = 2.5pt plus 1.25pt minus 0.5pt,
    itemsep = 0pt plus 1.25pt minus 0.5pt
  }

\theoremstyle{plain}
\newtheorem{theorem}{Theorem}
\newtheorem*{theorem*}{Theorem}
\newtheorem{lemma}[theorem]{Lemma}
\newtheorem{corollary}[theorem]{Corollary}
\newtheorem{conjecture}[theorem]{Conjecture}

\theoremstyle{definition}
\newtheorem{example}[theorem]{Example}
\newtheorem{definition}[theorem]{Definition}

\theoremstyle{remark}
\newtheorem{remark}[theorem]{Remark}

\numberwithin{theorem}{section}
\numberwithin{equation}{section}

\usepgfmodule{oo}

\DeclareRobustCommand\ShowAuthors[2]{%
  \ShowAuthorsSignal.emit({#1},{#2})%
}

\pgfoonew \ShowAuthorsSignal=new signal()

\DeclareRobustCommand\ShowAffiliations[1]{%
  \ShowAffiliationsSignal.emit(#1)%
}

\pgfoonew \ShowAffiliationsSignal=new signal()

\newcommand\Author[1]{
  \pgfoonew \CurrentPerson=new person()
  \CurrentPerson.set author(#1)
}

\newcommand\Email[1]{
  \CurrentPerson.set email(#1)
}

\newcommand\Address[1]{
  \CurrentPerson.set address(#1)
}

\newcommand\FirstPerson{0}
\newcommand\LastPerson{}

\pgfooclass{person}{
  \attribute author;
  \attribute email;
  \attribute address;

  \method person() {
    \pgfoothis.get id(\theid)

    \ifnum \FirstPerson=0
      \edef\FirstPerson{\theid}
    \fi

    \edef\LastPerson{\theid}

    \pgfoothis.get handle(\me)
    \ShowAuthorsSignal.connect(\me,show author)
    \ShowAffiliationsSignal.connect(\me,show affiliation)
  }

  \method get author(#1) {
    \pgfooget{author}{#1}
  }

  \method set author(#1) {
    \pgfooset{author}{#1}
  }

  \method get email(#1) {
    \pgfooget{email}{#1}
  }

  \method set email(#1) {
    \pgfooset{email}{#1}
  }

  \method get address(#1) {
    \pgfooget{address}{#1}
  }

  \method set address(#1) {
    \pgfooset{address}{#1}
  }

  \method show author(#1,#2) {%
    \pgfoothis.get id(\theid)%
    \pgfooget{author}{\theauthor}%
    \ifnum\theid=\FirstPerson%
    \else%
      \ifnum\theid=\LastPerson%
        #2%
      \else%
        #1%
      \fi%
    \fi%
    \mbox{\theauthor}%
  }

  \method show affiliation(#1) {%
    \pgfoothis.get id(\theid)%
    \pgfooget{author}{\theauthor}%
    \pgfooget{email}{\theemail}%
    \pgfooget{address}{\theaddress}%
    \ifnum\theid=\FirstPerson%
    \else%
      #1%
    \fi%
    \noindent\begin{minipage}{\linewidth}
      \noindent\begin{tabular}[t]{@{}l}
        \theauthor \quad \textsf{\theemail}\\
      \end{tabular}\\
      \theaddress%
    \end{minipage}%
  }
}

\newcommand\blfootnote[1]
  {%
    \begin{NoHyper}%
      \renewcommand\thefootnote{}%
      \footnote{#1}%
      \addtocounter{footnote}{-1}%
    \end{NoHyper}%
  }

\newcommand*\abs[1]{\left\lvert #1 \right\rvert}

\usepackage{mathtools}

\Author{Livia Campo}
\Address{Institut für Mathematik, Universität Wien, Oskar-Morgenstern-Platz 1, 1090 Wien, Austria}
\Email{livia.campo@univie.ac.at}

\Author{Tiago Duarte Guerreiro}
\Address{Departement Mathematik und Informatik, Universität Basel, Spiegelgasse 1, 4051 Basel, Switzerland}
\Email{tiago.duarteguerreiro@unibas.ch}

\Author{Erik Paemurru}
\Address{%
    Institute of Mathematics and Informatics, Bulgarian Academy of Sciences, Acad.\ G.~Bonchev Str.\ bl.~8, 1113, Sofia, Bulgaria.\\
    \textit{Former addresses:}
    Mathematik und Informatik, Universität des Saarlandes, 66123 Saarbrücken, Germany.\\
    Department of Mathematics, University of Miami, Coral Gables, Florida 33146, USA.%
}
\Email{algebraic.geometry@runbox.com}

\newcommand\Thanks{The first author was previously supported by the Korea Institute for Advanced Study (KIAS), grant No. MG087901. The second author was supported by Engineering and Physical Sciences
Research Council (EPSRC)  EP/V055399/1 and ERC StG Saphidir No. 101076412. The third author was supported by the Simons Investigator Award HMS, National Science Fund of Bulgaria, National Scientific Program ``Excellent Research and People for the Development of European Science'' (VIHREN), Project No.~KP-06-DV-7 and this is a contribution to Project-ID 286237555 -- TRR 195 -- by the Deutsche Forschungsgemeinschaft (DFG, German Research Foundation).}

\title{Blowups of smooth hypersurfaces, their birational geometry and divisorial stability}
\author{\ShowAuthors{, }{, }}
\date{4th~May 2026}
\newcommand\keywords{Fano hypersurfaces, Mori dream spaces, divisorial stability, birational models}
\newcommand\subjclass{14E30,14J30,14J40,14J45,14J70,32Q20}

\begin{document}

\maketitle

\begin{abstract}
Let $X$ be a smooth $n$-dimensional Fano hypersurface in $\mathbb P^{n+1}$ where $n \geq 3$.
Let $\Gamma$ be a smooth positive-dimensional complete intersection of~$X$, a hypersurface and one of more hyperplanes in $\mathbb P^{n+1}$.
Let $Y \to X$ be the blowup of $X$ along~$\Gamma$.
Let $\varphi \colon Y \rightarrow X$ be the blowup of $X$ along $\Gamma$. We describe the Mori chamber decomposition of $Y$ and its associated birational models. In particular, we show that $Y$ is a Mori dream space. We classify for which $X$ and $\Gamma$ the variety $Y$ is a Fano manifold and, if $X$ is a hyperplane, we classify the elementary Sarkisov links initiated by $\varphi$. Finally, we use this  Mori chamber decomposition above to prove that certain Fano manifolds as above do not admit a K\"ahler-Einstein metric.%
\blfootnote{\textup{2020} \textit{Mathematics Subject Classification}. \subjclass{}.}%
\blfootnote{\Thanks{}}%
\end{abstract}

\section{Introduction}

We work over the field of complex numbers. We focus on smooth $n$-dimensional hypersurfaces $X\subset \mathbb P^{n+1}$, where $n\geq 3$, containing a smooth $k$-dimensional hypersurface $\Gamma \subset \mathbb P^{k+1}$ for some $1 \leq k \leq n-2$, and we study the geometry of the blow up $Y$ of $X$ along $\Gamma$.
We give a complete and constructive description of the Nef, Movable, and Effective cones of $Y$ showing, in particular, that it is a Mori dream space. In addition, we describe a Cox ring of~$Y$. By the Lefschetz hyperplane section theorem $\mathrm{Pic}(Y) \simeq \mathbb Z^2$ is generated by the  pull-back of $\mathcal{O}_X(1)$ and the exceptional divisor of the blowup.
As a consequence, $Y$ admits exactly two Mori contractions: one to $X$, and another one to a birational model of $Y$ explicitly determined by the geometry of $X$ and $\Gamma$.

Mori dream spaces were introduced by Hu and Keel \cite{HuKeel} as natural geometric objects that are well-behaved with respect to the Minimal Model Program. For instance, the contractions of a Mori dream space are induced from those of an ambient quasi-smooth projective toric variety $T$ whose Mori chamber decomposition is a refinement of the Mori chamber decomposition of the Mori dream space (see \cite[Prop.~2.11]{HuKeel}). This refinement can be non-trivial, for instance there are examples in which the Nef cone of $T$ is strictly smaller than the Nef cone of the Mori dream space (cf \cref{lem:rd}). We give a concrete example of a flat deformation of smooth Fano fourfolds whose central fibre is a smooth weak Fano fourfold and whose Nef cone is a subcone of the Nef cone of the central fibre, see Example \ref{ex:jump}. See also \cite{jump} for different examples.

However, the property of being a Mori dream space is not automatically maintained when performing common birational transformations such as blowups. For instance, \cite{CastravetMDSblowups} examines the case of toric varieties (that are Mori dream spaces by \cite[Corollary~2.4]{HuKeel}) blown up at the identity point, providing examples that are Mori dream spaces, and others that are not.
A previous work of Ottem \cite{OttemHypersurfProducts} characterises when a normal $\mathbb{Q}$-factorial hypersurface   in products of projective spaces is a Mori dream space. In \cite{Ito}, the author gives a criterion for a normal projective variety of Picard rank 2 to be a Mori dream space. In particular, he shows that the blowup at a general point of a general Fano complete intersection of dimension at least three is a Mori dream space. However, his construction does not explicitly provide the birational models induced by the Mori chamber decomposition of the blowup.

Indeed, the property of being a Mori dream space is closely related to finding birational models (cf \cite{OkawaImagesMDS, AbbanZucconiMDS, CampoDuarteGuerreiro, BlancLamyWeak,AbbanCheltsovPark3folds,Pae24Sextic,OkadaSolidFanoHypersurf}), also in relation to moduli problems \cite{CastravetTevelevMnoMDS}, and to those of rigidity (\cite{KollarRigidity10names,deFernexBirRigHypersurf, KOPP24,OkadaMFSI}).
Our work builds up on these in spirit and structure (especially \cite{OttemHypersurfProducts} and \cite{Ito}). Our main theorem is the following.

\newcommand\MainCones{
    Let $\Pi \cong \mathbb P^{k+1}$ be a linear subspace of $\mathbb P^{n+1}$, where $1 \leq k \leq n-2$.
    Let $X\subset \mathbb P^{n+1}$ be a smooth hypersurface and $\Gamma \subset \Pi$ a smooth hypersurface of $\Pi$ contained in~$X$.
    Let $\varphi \colon Y \rightarrow X$ be the
    blow up of  $X$ along $\Gamma$. Let $E$ be the $\varphi$-exceptional divisor and $H$ the pullback of a hyperplane section of~$X$.
    Then,
    \begin{align*}
        \mathrm{Nef}(Y)&= \begin{cases}
        \mathbb{R}_+[H]+\mathbb{R}_+[(\deg \Gamma) H-E] & \text{if\,\,\, $\deg \Gamma \not = \deg X$ or $\Pi \subset X,$} \\
        \mathbb{R}_+[H]+\mathbb{R}_+[H-E] & \text{if\,\,\, $\deg \Gamma = \deg X$  and $\Pi \not \subset X$,}
        \end{cases}\\
        \mathrm{Mov}(Y)&=\begin{cases}
        \mathbb{R}_+[H]+\mathbb{R}_+[(\deg \Gamma) H-E] & \text{if $X$ is a hyperplane and $\mathrm{codim_X\Gamma}=2$,} \\
        \mathbb{R}_+[H]+\mathbb{R}_+[H-E] & \text{otherwise,}
        \end{cases} \\
        \mathrm{Eff}(Y)&=\mathbb{R}_+[E]+\mathbb{R}_+[H-E].
    \end{align*}
    In particular, $Y$ is a Mori dream space with Cox ring $\mathbb C[u, x_0, \ldots, x_{n+1}, z] / I_Y$ where the ideal $I_Y$ is defined in \cref{lem:stricttransf}. Moreover, $\varphi$ induces a birational map to a
    \begin{itemize}
        \item  Fano fibration if $3-\deg X+\dim \Gamma >0$,
        \item  Calabi-Yau fibration if $3-\deg X+\dim \Gamma =0$,
        \item  fibration into canonically polarised varieties if $3-\deg X+\dim \Gamma <0$.
    \end{itemize}
}

\begin{theorem*} [= Theorem \ref{thm:main}]
    \MainCones{}
\end{theorem*}
We additionally determine explicitly the birational models of $Y$. Moreover, we determine under which conditions $Y$ is a Fano variety assuming $X$ is a Fano variety, see Theorem \ref{thm:mainII}.

\newcommand\MainFano{
    Let $\Pi \cong \mathbb P^{k+1}$ be a linear subspace of $\mathbb P^{n+1}$, where $1 \leq k \leq n-2$.
    Let $X\subset \mathbb P^{n+1}$ be a smooth Fano hypersurface and $\Gamma \subset \Pi$ a smooth hypersurface of $\Pi$ contained in~$X$. Let $\varphi \colon Y \rightarrow X$ be the
    blow up of $X$ along $\Gamma$.
    Then $Y$ is a Fano variety if and only if one of the following holds:
    \begin{enumerate}
        \item $\deg X = \deg \Gamma$, $\Pi \not \subset X$ and $\deg \Gamma \leq 2+\dim \Gamma$, or
        \item all of the following hold:
        \begin{enumerate}[label=(\theenumi.\arabic*), ref=\theenumi.\arabic*]
        \item $\deg X \neq \deg \Gamma$ or $\Pi \subset X$,
        \item $\deg X \leq \dim X+1-\deg \Gamma \cdot (\mathrm{codim}_X\Gamma-1)$, and
        \item if $\Pi \subset X$ then $\dim \Gamma \leq \operatorname{codim}_X \Gamma - 2$.
        \end{enumerate}
    \end{enumerate}
    In particular, if $Y$ is a Fano variety, then $-K_{\Gamma}$ is nef.
}

The following theorem classifies  Sarkisov links initiated from the blowup of $\mathbb P^n$ along $\Gamma$.

\begin{theorem*} [= Theorem \ref{thm: Sarkisovlink}]
    Suppose $X = \mathbb P^n$.
    Let $\Pi \cong \mathbb P^{k+1}$ be a linear subspace of $X$, where $1 \leq k \leq n-2$.
    Let $\varphi \colon Y \rightarrow X$ be the blowup of $X$ along a smooth hypersurface $\Gamma \subset  \Pi$ contained in~$X$. Then $\varphi$ initiates a Sarkisov link if and only if both of the following hold:
    \begin{enumerate}
    \item $-K_{\Gamma}$ is nef, and
    \item $\mathrm{codim}_{X}\Gamma\leq2$ or $\dim \Gamma \geq 2$.
    \end{enumerate}
\end{theorem*}

Divisorial stability was introduced in \cite{Kentovolume} as a weaker notion of K-stability. Its importance comes from the seminal work in \cite{CDS1,CDS2,CDS3} proving the equivalence of the existence of K\"ahler-Einstein (KE) metrics on Fano manifolds and the algebraic-geometric notion of K-(poly)stability. We prove non-existence of KE metrics on blowups of projective spaces and quadrics along a line, generalising \cite[Lemma 3.22]{CalabiBook}.

\newcommand\MainStability{
    Let $\Gamma$ be a projective line in a smooth Fano hypersurface. Let $\varphi \colon Y \rightarrow X$ be the blowup of $X$ along $\Gamma$ and suppose that $Y$ is a Fano variety. Then, if $X=\mathbb P^n$ or $X$ is a quadric, then $Y$ is divisorially unstable. In particular, $Y$ is not K-stable and does not admit a K\"ahler-Einstein metric.
}

\begin{theorem*} [= Theorem \ref{thm:stabilityI}]
    \MainStability{}
\end{theorem*}

We also show in \cref{thm:divisorial unstability 1000} that when $\Gamma =\mathbb P^k$, then $Y$ is divisorially unstable for high enough $k$ assuming the dimension of $X$ is bounded by $1000$. Moreover, in \cref{conj:divisorial unstability} we predict that this holds for any dimension $n\geq k+2$.

\paragraph{Acknowledgements.} We would like to thank Igor Krylov for answering a question on intersection numbers and Calum Spicer for the question leading to \cref{exa:flop from smooth to singular}. The second author wishes to thank Paolo Cascini for the hospitality provided at Imperial College London during the course of this work.

\section{Preliminaries}

\begin{definition} [{\cite[Definition~1.10]{HuKeel}}]
    A normal projective variety $Y$ is a \textbf{Mori dream space} if the following conditions hold.
    \begin{enumerate}
        \item $Y$ is $\mathbb{Q}$-factorial and $\operatorname{Pic}(Y)$ is finitely generated.\label{first MDS condition}
        \item $\operatorname{Nef}(Y)$ is the affine hull of finitely many semi-ample line bundles. \label{second MDS condition}
        \item There is a finite collection of small $\mathbb{Q}$-factorial modifications $f_i \colon Y \dashrightarrow Y_i$ such that each $Y_i$ satisfies \ref{first MDS condition} and \ref{second MDS condition} and $\overline{\operatorname{Mov}}(Y)= \bigcup f_i^*(\operatorname{Nef}(Y_i))$.
    \end{enumerate}
\end{definition}

Throughout the paper, $X, \Gamma, \Pi$, $\Pi'$, $n$, $k$, $d$ and $r$ are as in \cref{set:hyp}. We recall that a $\mathbb Q$-factorial toric variety is a Mori dream space \cite[Corollary~2.4]{HuKeel}. Indeed they have Cox rings which are polynomial algebras generated by sections corresponding to the 1-dimensional rays of the defining fan.

\subsection{Setting} \label{set:hyp}
Let $\Pi$ be a $(k+1)$-dimensional linear subspace of $\mathbb P^{n+1}$
 such that $2 \leq \dim \Pi \leq n-1$, where $n\geq 3$. Let $\Pi'$ be  a maximal dimensional linear subspace of $\mathbb P^{n+1}$ that does not intersect $\Pi$. Let $\Gamma$ be a smooth hypersurface of $\Pi$ of degree $r$ and $X$ be an $n$-dimensional smooth hypersurface of $\mathbb P^{n+1}$ of degree $d$ containing~$\Gamma$. Consider $\varphi\colon Y\rightarrow X$ be the blowup of $X$ along $\Gamma$. To summarise, we have
\[
   \begin{aligned}
   \dim X & = n, & \qquad
   \deg X & = d,\\
   \dim \Gamma & = k, & \qquad
   \deg \Gamma & = r.
   \end{aligned}
\]
We have the linear equivalences
\[
    \begin{aligned}
    -K_X & \sim \mathcal{O}_X(n+2-d), &
    \\
    -K_{\Gamma} & \sim \mathcal{O}_{\Gamma}(k+2-r). &
    \end{aligned}
\]
We denote the variables on $\mathbb P^{n+1}$ by $x_0, \ldots, x_{n+1}$. After a suitable change of coordinates, we have
\[
\begin{aligned}
    \Pi & = (x_{k+2}=\cdots =x_{n+1}=0) \simeq \mathbb P^{k+1} \subset \mathbb P^{n+1},\\
    \Pi' & = (x_{0}=\cdots =x_{k+1}=0) \simeq \mathbb P^{n-k-1} \subset \mathbb P^{n+1},\\
    \Gamma & = (h_{r}(x_0,\ldots,x_{k+1})=0) \subset \Pi \subset \mathbb P^{n+1},\\
    X & = (f_{d}(x_0,\ldots,x_{n+1})=0) \subset \mathbb{P}^{n+1},
\end{aligned}
\]
where $h_r \in \mathbb{C}[x_0,\ldots,x_{k+1}]$ and $f_{d} \in \mathbb C[x_0,\ldots,x_{n+1}]$ are homogeneous polynomials of respectively degrees $r$ and~$d$. Since $X$ contains $\Gamma$, there exist homogeneous degree $d-1$ polynomials $F_i \in \mathbb C [x_0,\ldots,x_{n+1}]$ and a homogeneous degree $d-r$ polynomial $g_{d-r} \in \mathbb C[x_0,\ldots,x_{k+1}]$ such that
\[
    f_{d}=x_{k+2}F_{k+2}+\cdots+ x_{n+1}F_{n+1} + h_rg_{d-r}.
\]
If $r>d$, then $g_{d-r}$ is necessarily the zero polynomial. If $X \cong \mathbb P^n$, then we assume $X\colon (x_{n+1}=0)$.

\begin{lemma} \label{lem:sing}
    Both of the following hold:
    \begin{enumerate}[label=(\alph*), ref=\alph*]
    \item \label{itm:smooth implies ineq} if $d \geq 2$ and $\Pi \subset X$, then $\dim \Pi \leq \operatorname{codim}_{\mathbb P^{n+1}} \Pi - 1$, and
    \item \label{itm:ineq implies general is smooth} if $\dim \Pi \leq \operatorname{codim}_{\mathbb P^{n+1}} \Pi - 1$, then a general
    hypersurface of $\mathbb P^{n+1}$ of degree $d$ containing $\Pi$ is smooth.
    \end{enumerate}
\end{lemma}
\begin{proof}
  Since $\Pi \subset X$ we have
    $$
    X = (x_{k+2}F_{k+2}+\cdots + x_{n+1}F_{n+1}=0) \subset \mathbb P^{n+1}.
    $$

    \labelcref{itm:smooth implies ineq}
    We take partial derivatives restricted to the locus $(x_{k+2}=\cdots =x_{n+1}=0)$. The Jacobian is
    $$
    \begin{pmatrix}
        0 & \ldots & 0 & F_{k+2} & \ldots & F_{n+1}
    \end{pmatrix},
    $$
    where, by assumption, the polynomials $F_i$, $k+2\leq i\leq n+1$ are homogeneous of positive degree. Hence, $X$ is singular along $(x_{k+2}=\cdots =x_{n+1}=F_{k+2} = \ldots = F_{n+1}=0)$. That is, along
    $$
    S \colon (F_{k+2} = \ldots = F_{n+1}=0) \subset \Pi.
    $$
    The locus $S$ is non-empty if and only if its dimension is non-negative. We have
    $$
    \mathrm{dim} S \geq k+1-(n+1-(k+2)+1)=2k-n+1
    $$
    and the claim follows.

    \labelcref{itm:ineq implies general is smooth} By Bertini's theorem, the singularities of a general hypersurface of degree $d$ containing $\Pi$ are on the hyperplane~$\Pi$. On the other hand, the hypersurface $(x_{k+2} x_0^{d-1} + \ldots + x_{n+1} x_{k+1}^{d-1}=0)$ contains and is smooth along~$\Pi$.
\end{proof}

\begin{corollary} \label{cor:sing}
    Let $\Pi \cong \mathbb P^{k+1}$ be a linear subspace of $\mathbb P^{n+1}$, where $1 \leq k \leq n-2$.
    Let $X\subset \mathbb P^{n+1}$ be a smooth hypersurface and $\Gamma \subset \Pi$ a smooth hypersurface of $\Pi$ contained in~$X$.
    If $\mathrm{dim}\Gamma \geq \mathrm{codim}_{X}\Gamma - 1$, then $\deg  \Gamma \leq \deg X$ or $X$ is a hyperplane.
\end{corollary}

\begin{proof}
    Follows from \cref{lem:sing} using $\dim \Pi = \dim \Gamma + 1$ and $\operatorname{codim}_{\mathbb P^{n+1}} \Pi = \operatorname{codim}_X \Gamma$.
\end{proof}

\section{Birational models}

We remind that throughout the paper, $X, \Gamma, \Pi$, $\Pi'$, $n$, $k$, $d$ and $r$ are as in \cref{set:hyp}.

\subsection{Birational models of the ambient space} \label{sec:bir models of ambient space}

We define $\mathbb P := \mathrm{Proj} \big(\mathbb C [x_0,\ldots,x_{n+1},z]\big) \simeq \mathbb P(1^{n+2},r)$ where $\deg z = r$. Let $\Phi \colon T \longrightarrow \mathbb P$ be the blowup of $\mathbb P$ along $(x_{k+2}=\cdots=x_{n+1}=z=0) \cong \Pi$. Let $H$ be the pullback of a hyperplane in $\mathbb P$ and $E$ the $\Phi$-exceptional divisor. Then $T$ is a rank 2 toric variety with
$$
\mathrm{Pic}(T)=\mathbb Z [H]\oplus \mathbb Z [E]
$$
and $E\simeq \mathbb P(\mathcal{E})$ where $\mathcal{E}$ is the vector bundle  $\mathcal{O}_{\mathbb P^{k+1}} \oplus \mathcal{O}_{\mathbb P^{k+1}}(r-1)^{\oplus (n-k)}$.
Define
\[
\begin{aligned}
    \Pi_{\mathbb P} &= (x_{k+2}=\cdots=x_{n+1} = 0)\subset \mathbb P,\\
    \Pi_{\mathbb P}' &= (x_{0}=\cdots=x_{k+1} = 0)\subset \mathbb P.
\end{aligned}
\]
\newcommand\PA{\ensuremath{\widetilde{\Pi_{\mathbb P}}}}
\newcommand\PB{\ensuremath{\widetilde{\Pi_{\mathbb P}'}}}
Denote by $\PA$ and $\PB$ the strict transforms of $\Pi_{\mathbb P}$ and $\Pi_{\mathbb P}'$, respectively.

\begin{lemma} \label{lem:conesToric}
    The cone of effective divisors of $T$ is generated by $E$ and $H-E$. Moreover,
    $$
    \mathrm{Nef}(T)=\mathbb R_+[H]\oplus \mathbb R_+[rH-E] \subseteq
    \mathbb R_+[H]\oplus \mathbb R_+[H-E] = \mathrm{Mov}(T).
     $$
\end{lemma}

\begin{proof}
    This follows from
    \cite[Theorem 15.1.10]{CoxToricVarieties}.
\end{proof}
There are two maps from $T$ as in the following diagram

\[ \begin{tikzcd}
T \arrow[rd, "\alpha"] \arrow[d,swap,"\Phi"]  \\%
\mathbb P & Q
\end{tikzcd}
\]
where $Q$ is the image of $\alpha$. The morphism $\Phi$ is given by the linear system $|mH|$ and the morphism $\alpha$ is given by the linear system $|m(rH-E)|$ for some big enough~$m$.
Note that if $r > 1$, then $rH - E$ is in the interior of the effective cone by \cref{lem:conesToric}, therefore big.
By Lemma \ref{lem:conesToric}, the divisor $H - E$ is in the boundary of the effective cone, therefore not big. If $r > 1$, then $rH - E$ is movable, therefore $\alpha$ does not contract any divisor. So the morphism $\alpha$ is a fibration over a lower-dimensional variety if $r=1$ and a small contraction whenever $r>1$. We will write these maps explicitly using the Cox coordinates of $T$,
$$
\mathrm{Cox}(T) = \bigoplus_{(m_1,m_2) \in \mathbb{Z}^2} H^0(T,m_1H+m_2E) \; .
$$

Since $T$ is toric, the Cox ring of $T$ is isomorphic to a (bi)-graded polynomial ring. In this case, defining
\[
\mathbb A^{n+3} = \operatorname{Spec} \mathbb C[u, x_0, \ldots, x_{k+1}, z, x_{k+2}, \ldots, x_{n+1}],
\]
$T$ is the geometric quotient
\[
    \frac{\mathbb A^{n+3} \setminus V\bigl((u, x_0, \ldots, x_{k+1}) \cap (z, x_{k+2}, \ldots, x_{n+1})\bigr)}{\mathbb C^* \times \mathbb C^*},
\]
where the $\mathbb C^* \times \mathbb C^*$-action is given by the matrix
\begin{equation} \label{eqn:matrix}
\left( \begin{array}{ccccccccccc}
u & x_0 & \cdots & x_{k+1} & z & x_{k+2} & \cdots & x_{n+1}\\
0 & 1 & \cdots & 1 & r & 1 & \cdots & 1\\
-1 & 0 & \cdots & 0 & 1 & 1 & \cdots & 1
\end{array}
\right)
\end{equation}
meaning that for all $(\lambda, \mu) \in \mathbb C^* \times \mathbb C^*$, we have
\[
(\lambda, \mu) \cdot (u, x_0, z, x_1, \ldots, x_{n+1}) = (\mu^{-1} u, \lambda x_0, \lambda^{r} \mu z, \lambda \mu x_1, \ldots, \lambda \mu x_{n+1}).
\]
Note that we have ordered the column vectors in the matrix \labelcref{eqn:matrix} in anticlockwise order.
The Cox ring of $T$ is the $\mathbb Z^2$-graded polynomial ring
\begin{equation*}
\mathrm{Cox}(T) \simeq \mathbb{C}[u, x_0, \ldots, x_{k+1}, z, x_{k+2}, \ldots, x_{n+1}],
\end{equation*}
where again the grading comes from the matrix \labelcref{eqn:matrix}, meaning that the bidegree of a variable is the corresponding column vector in the matrix \labelcref{eqn:matrix}.
Note that $u \in H^0(T,\mathcal{O}_T(E))$. The map $\Phi\colon T \rightarrow \mathbb P$ is given in coordinates by
\[
\begin{aligned}
T & \to \mathbb P = \operatorname{Proj} \mathbb C[x_0, \ldots, x_{k+1}, z, x_{k+2}, \ldots, x_{n+1}]\\
(u, x_0, \ldots, x_{k+1}, z, x_{k+2}, \ldots, x_{n+1}) & \mapsto (x_0, \ldots, x_{k+1}, u z, u x_{k+2}, \ldots, u x_{n+1}).
\end{aligned}
\]

\begin{lemma} \label{lem:toric}
The blowup $\Phi$ induces a  map $\Psi$ such that
\begin{itemize}
    \item if $r=1$, the map  $\Psi=\alpha$ is a $\mathbb P^{k+2}$-bundle over $\mathbb{P}^{n-k}$,
    \item if $r>1$, the map $\Psi$ is the composition of
    \begin{enumerate}
        \item a small $\mathbb{Q}$-factorial modification $\Theta \colon T \dashrightarrow T'$ which contracts $\PA\simeq \mathbb P(1^{k+2},r)$ to a point and extracts $E \cap \PB \simeq \mathbb P^{n-k-1}$ and is an isomorphism otherwise. This operation defines $T'$, namely $T'$ is the geometric quotient
        \[
        \frac{\mathbb A^{n+3} \setminus V\bigl(u, x_0, \ldots, x_{k+1}, z) \cap (x_{k+2}, \ldots, x_{n+1})\bigr)}{\mathbb C^* \times \mathbb C^*},
        \]
        where $\mathbb A^{n+3}$ has the variables $u, x_0, \ldots, x_{k+1}, z, x_{k+2}, \ldots, x_{n+1}$ and the $\mathbb C^* \times \mathbb C^*$-action is given by the matrix \labelcref{eqn:matrix}.
        The variety $T'$ is singular along $E \cap \PB$ and each point is a cyclic quotient singularity of type
        $$\frac{1}{r-1}(\underbrace{1,\ldots,1}_{k+3},\underbrace{0,\ldots,0}_{n-k-1})$$
        followed by
        \item a $\mathbb{P}(1^{k+3},r-1)$-bundle over $\mathbb P^{n-k-1}$.
    \end{enumerate}
\end{itemize}
\end{lemma}

\begin{proof}
If $r=1$, $\alpha \colon T \rightarrow Q$ is a fibration given by
$$(u,x_0,\ldots,x_{n+1},z) \longmapsto (x_{k+2},\ldots,x_{n+1},z).$$
and $T'\simeq \mathbb P^{n-k}$. Moreover, the fibres are isomorphic to $\mathbb P^{k+2}$.

 On the other hand suppose $r>1$. Then $rH-E$ is big and in the interior of the movable cone of $T$. Therefore $\Theta$ decomposes as
 \begin{equation} \label{eqn:alpha}
 \begin{tikzcd}
T \arrow[rr, "\Theta",dashed] \arrow[dr,swap,"\alpha"] && T' \arrow[dl,swap,"\alpha'",swap] \\%
& Q &
\end{tikzcd}
\end{equation}
where $\alpha$ is given by the complete linear system $\abs{rH-E}$ on $T$ and similarly $\alpha'$ is given by the complete linear system $\abs{rH-E}$ on $T'$. Notice that $T$ and $T'$ have the same Cox ring but slightly different irrelevant ideals, namely $I_T = (u,x_0,\ldots,x_{k+1}) \cap (z,x_{k+2},\ldots,x_{n+1})$ and
$I_{T'} = (u,x_0,\ldots,x_{k+1},z) \cap (x_{k+2},\ldots,x_{n+1})$.
Notice that $\alpha$ is given by all monomials of bidegree $\binom{r}{1}$. It is then clear that $\alpha$ contracts  $\PA \colon (x_{k+2}=\ldots=x_{n+1}=0)\subset T$ and $\alpha'$ contracts $E \cap \PB \colon (u=x_{0}=\ldots=x_{k+1}=0)\subset T'$ to the same point and that these are isomorphisms away from these loci. Hence, $\Theta$ swaps $\PA$ with $E\cap \PB$. This introduces some cyclic quotient singularities in general. Indeed, any point in $E\cap \PB$ is covered by the $n-k$ affine patches
$$U_{zx_i}:=(zx_i\not =0)=\mathrm{Spec}\big(\mathbb C[u,x_0,\ldots,x_{k+1},z,x_{k+2},\ldots,\widehat{x_i},\ldots,x_{n+1},z^{-1},x_i^{-1}]\big)^{\mathbb C^*\times \mathbb C^*}$$
where $k+2\leq i\leq n+1$.  The $\mathbb C$-algebra of $\mathbb C^*\times \mathbb C^*$-invariant regular functions of $U_{zx_i}$ is generated by
$$
\frac{f_{r-1}(ux_i,x_0,\ldots,x_{k+1})x_i}{z}, \quad \frac{x_j}{x_i},\,\, k+2 \leq j\not = i\leq n+1
$$
where $f_{r-1}$ is any function of degree $r-1$. On the other hand consider the action $\mathbf{\mu}_{r-1} \curvearrowright \mathbb A^{n+2}$ of the multiplicative cyclic group of order $r-1$ on $\mathbb A^{n+2}$, given by
$$
\epsilon\cdot (u,x_0,\ldots,\widehat{x_i},\ldots,x_{n+1}) = (\epsilon u,\epsilon x_0,\ldots,\epsilon x_{k+1},x_{k+2},\ldots,\widehat{x_i},\ldots,x_{n+1})
$$
where $\epsilon$ is a primitive $r-1$-root of unity. Then $\mathrm{Spec}\big(\mathbb C[u,x_0,\ldots,x_{k+1}]\big)^{\mathbf{\mu}_{r-1}}$ is the affine variety given by the vanishing of the polynomials $f_{r-1}(u,x_0,\ldots,x_{k+2})$. We showed that
$$
U_{zx_i} \simeq \frac{1}{r-1}(\underbrace{1,\ldots,1}_{k+3},\underbrace{0,\ldots,0}_{n-k-1}).
$$
Hence, $T'$ is singular along $E\cap \PB$ and each of its points is a cyclic quotient singularity as above.  Moreover $\mathrm{codim}_{T'}E\cap \PB=n+2-(n-k-1)=k+3\geq 4$.

The nef cone of $T'$ is generated by $rH-E$ and $H-E$, see
\cite[Theorem 15.1.10]{CoxToricVarieties}. The divisor $H-E$ is not big and so it induces a fibration $\Phi' \colon T' \rightarrow \PA$ given in coordinates by
$$(u,x_0,\ldots,x_{n+1},z) \longmapsto (x_{k+2},\ldots,x_{n+1}).$$ The fibres of $\Phi'$ are isomorphic to $\mathbb P(1^{k+3},r-1)$.
\end{proof}

\subsection{\texorpdfstring
    {Birational models of $X$}
    {Birational models of X}
}

We recall that $X, \Gamma, \Pi$, $\Pi'$, $n$, $k$, $d$ and $r$ are as in \cref{set:hyp}.
The morphism $\Phi$, the weighted projective $\mathbb P$, $\Pi_{\mathbb P}$ and $\Pi_{\mathbb P}'$ are as in \cref{sec:bir models of ambient space}.
    We embed $\mathbb P^{n+1}$ in $\mathbb P$ as $(z - h_r= 0)$ and $X$ in $\mathbb P$ as
\[
X \colon (z - h_r= f_{d}(x_0,\ldots,x_{n+1})=0) \subset \mathbb{P}.
\]
The variety $Y$ is the strict transform of $X$ under $\Phi$.
The blowup $\varphi \colon Y \rightarrow X$ of $X$ along $\Gamma$ is the restriction $\Phi\vert_Y$. We again denote by $E$ the $\varphi$-exceptional divisor. In the same way, we again denote by $H$ a pull-back of a hyperplane section of $X$ under $\varphi$.

\begin{lemma} \label{lem:stricttransf}
    The variety $Y$ is given by the ideal
    \[
    I_Y = (x_{k+2}F^{'}_{k+2}+\cdots +x_{n+1}F^{'}_{n+1} + zg_{d-r},~ zu-h_r) \subseteq \mathbb C[u, x_0, \ldots, x_{n+1}, z]
    \]
    where
    \[
    F^{'}_j(u,x_0,\ldots,x_{n+1},z) :=F_j(x_0,\ldots,x_{k+1},ux_{k+2},\ldots,ux_{n+1},uz).
    \]
\end{lemma}
\begin{proof}
   The hypersurface $X$ is isomorphic to the codimension 2 complete intersection
    \begin{align*}
        x_{k+2}F_{k+2}+\cdots x_{n+1}F_{n+1} + h_rg_{d-r}&=0 \\
        z-h_r&=0
    \end{align*}
    Since $\Gamma$ is smooth and isomorphic to $V(h_r) \subseteq \mathbb P^{k+1}$, we find that $h_r$ is irreducible. Denote
    \[
    f' := x_{k+2}F^{'}_{k+2}+\cdots +x_{n+1}F^{'}_{n+1} + zg_{d-r}.
    \]
    To show that $Y$ is the strict transform of $X$ with respect to~$\Phi$, it suffices to show that the ideal $(f', zu-h_r)$ is saturated with respect to~$u$.

    Suppose that the ideal $(f', zu-h_r)$ of $\mathbb C[u, x_0, \ldots, x_{n+1}, z]$ is not saturated with respect to~$u$. Then there exists an element $K$ in the saturation of the ideal $(f', zu-h_r)$ that is not in $(f', zu-h_r)$. More precisely, there exist a positive integer $m$ and homogeneous polynomials $G, H, K \in \mathbb C[u, x_0, \ldots, x_{n+1}, z]$ such that $K$ is not in the ideal $(f', zu-h_r)$ and we have the equality
    \[
    f'G + (zu-h_r)H = u^m K.
    \]
    Without loss of generality, $u$ does not divide $K$ and $m$ is the smallest possible.

    We show that $G \notin (u, zu-h_r) = (u, h_r)$. Assume there exist homogeneous polynomials $G_1, G_2 \in \mathbb C[u, x_0, \ldots, x_{n+1}, z]$ such that
    \[
    G = u G_1 + (z u - h_r) G_2.
    \]
    Then
    \[
    u^m K = u f' G_1 + (zu - h_r) (H + f' G_2).
    \]
    We see that $u$ divides $H + f' G_2$. This contradicts $m$ being the smallest possible. Therefore, $G \notin (u, zu-h_r) = (u, h_r)$.

    Since the ideal $(u, h_r)$ is prime, we find that $f' \in (u, h_r)$. Let
    \[
    \Phi_{\mathrm{alg}}\colon \mathbb C[x_0, \ldots, x_{n+1}, z] \to \mathbb C[u, x_0, \ldots, x_{n+1}, z]
    \]
    be the $\mathbb C$-algebra homomorphism corresponding to~$\Phi$. Note that the image of $\Phi_{\mathrm{alg}}$ is in the $\mathbb C$-subalgebra given by
    \[
    \mathbb C[x_0, \ldots, x_{n+1}, z, ux_{k+1}, \ldots, ux_{n+1}, uz].
    \]
    Since $\Phi_{\mathrm{alg}}(f + (z-h_r)g_{d-r}) = uf' \in (u, h_r)^2$, we find that $f + (z-h_r)g_{d-r}$ is in the ideal $(x_{k+1}, \ldots, x_{n+1}, z, h_r)^2$. Applying the homomorphism $z \mapsto h_r$, we see that $f \in (x_{k+2}, \ldots, x_{n+1}, h_r)^2$. Therefore, $X$ is singular along~$\Gamma$, a contradiction. This shows that $Y$ is the strict transform of~$X$.
\end{proof}

The following lemma is classical.

\begin{lemma} \label{lem:r1}
    Suppose $r=1$.
    \begin{enumerate}
        \item If $X=\mathbb P^n$, then the projection $\pi \colon X \dashrightarrow \mathbb P^{n-k-1}$ away from $\Pi$ can be decomposed into the blowup of $\Gamma$ followed by a $\mathbb P^{k+1}$-bundle over $\mathbb P^{n-k-1}$.
    \item If $X \not = \mathbb P^n$, then the projection $\pi \colon X \dashrightarrow \mathbb P^{n-k}$ can be decomposed as the blowup of $\Gamma \subset X$ followed by a fibration to $\mathbb P^{n-k}$ whose general fibre $F$ is isomorphic to a smooth hypersurface such that $K_F\sim \mathcal{O}_F(d-k-3)$.
    \end{enumerate}

\end{lemma}
\begin{proof}
     By Lemma \ref{lem:toric}, we know that $\Theta$ is a $\mathbb P^{k+2}$-bundle over $\mathbb{P}^{n-k}$ given by the linear system $|H-E|$. If $X=\mathbb P^n$, then $Y$ is a divisor in $|H-E|$. Then, $\Theta\vert_Y$ restricts to a $\mathbb P^{k+1}$-bundle over $\mathbb P^{n-k-1}$.

Suppose $X\not = \mathbb P^n$. Then, $\Theta\vert_Y \colon Y \rightarrow \mathbb{P}^{n-k}$ is a fibration and the general fibre $F$ of $\theta$ is isomorphic to
$$
F^{''}_{k+2} + \cdots + F^{''}_{n+1} + g_{d-r}=0,
$$
where
    $F^{''}_{i} = F^{'}_{i}(h_r, x_0, \ldots, x_{k+1}, x_{k+2} h_r, \ldots, x_{n+1} h_r)$
That is, a degree $d-r=d-1$ smooth
hypersurface $S$ in $\mathbb P ^{k+1}$ (by generic smoothness, see \cite[Proposition~III.10.7]{Hartshorne}). By adjunction $K_F\sim \mathcal{O}_F(d-1-k-2) = \mathcal{O}_F(d-k-3)$.
\end{proof}

The following lemma characterises the restrictions of $\Theta$ to $Y$ and is the technical core of the paper.

\begin{lemma} \label{lem:restr}
    Let $r>1$. Then $\Theta\vert_Y $ is either
    \begin{enumerate}[label=(\alph*), ref=\alph*]
        \item \label{item:div} a divisorial contraction if and only if $X=\mathbb P^n$ and $\mathrm{codim}_X\Gamma=2$ in which case $\PA \cap Y$ is a divisor contracted to
        $$
        \mathbf{p}_z \in Z\colon (zu-h_r = 0) \subset \mathrm{Proj}\big(\mathbb C[u,x_0,\ldots,x_{n-1},z]\big ) = \mathbb P (1^{n+1},r-1).
        $$
        Moreover, $\mathbf{p}_z$ is the cyclic quotient singularity $$\mathbf{p}_z \sim \frac{1}{r-1}(\underbrace{1,\ldots,1}_{n}).$$ In particular $Z$ is terminal if and only if $-K_{\Gamma}$ is nef. It is canonical but not terminal if and only if $K_{\Gamma} \sim \mathcal{O}_{\Gamma}(1)$.
        \item  \label{item:sqm} a (non-isomorphism) small $\mathbb Q$-factorial modification $Y \dashrightarrow Y'$ if and only if either
        \begin{enumerate}
            \item $r\neq d$ provided that $X \not =\mathbb P^n$ or $\mathrm{codim}_X\Gamma>2$ or
            \item $r=d$ provided that $\Pi \subset X$;
        \end{enumerate}
        where the morphisms $\alpha$ and $\alpha'$ of diagram~\labelcref{eqn:alpha} contract $\PA \cap Y$ and $E \cap \PB \cap Y'$ to a point, respectively.
        Moreover $\mathrm{codim}_{Y'}E\cap  \PB=k+1$.
        \item \label{item:iso} an isomorphism otherwise, that is, when $r=d$ and $X$ does not contain $\Pi$.
    \end{enumerate}
\end{lemma}

\begin{proof}
     By Lemma \ref{lem:toric}, we know that there is an isomorphism in codimension one
         \begin{equation*}
         \begin{tikzcd}
        T \arrow[rr, "\Theta",dashed] \arrow[dr,swap,"\alpha"] && T' \arrow[dl,swap,"\alpha'",swap] \\%
        & Q &
        \end{tikzcd}
        \end{equation*}
        where $\alpha$ contracts the locus $\PA$ and nothing else. By Lemma \ref{lem:stricttransf}, we have
 $$
 \PA \cap Y \colon (zg_{d-r}=zu-h_r=0) \subset T.
 $$
 Since $I_T = (u,x_0,\ldots,x_{k+1}) \cap (z,x_{k+2},\ldots,x_{n+1})$ we can assume $z=1$ at $\PA$. Hence, $\PA \cap Y$ is isomorphic to
 $$
  (g_{d-r}=0) \subset \mathbb{P}^{k+1}.
 $$
 Notice that, a priori, there is no reason for $g_{d-r}$ not to be identically zero.
 Let $\alpha\vert_Y$ be the contraction of $\PA \cap Y$ to a point, which is induced by the linear system $|m(rH-E)|$ for $m$ big enough. Notice that either $\dim \PA \cap Y\geq 1$ or $\PA \cap Y$ is empty. So, if $\alpha\vert_Y$ is a finite map, it follows that it is an isomorphism and $\PA \cap Y$ is empty.

 Suppose that $\alpha\vert_Y$ is not small nor an isomorphism. Then $\alpha\vert_Y$ is divisorial since $rH-E$ is a big divisor on $Y$ for $r>1$. Suppose $\PA \cap Y$ is a divisor in $Y$. Then, in particular, $g_{d-r}$ is identically zero and $k+1=\dim \PA \cap Y =n-1$, that is, $\mathrm{codim}_X\Gamma=2$. By Corollary \ref{cor:sing}, it follows that $X$ is a hyperplane. Consider the projection $ \pi \colon Z \dashrightarrow X$ away from $\mathbf{p}_z$. Then, the map $\pi$ can be resolved by a single blowup to $Y$ with exceptional divisor  $\PA$. Hence, the following diagram commutes
     \[
     \begin{tikzcd}
   & \PA \subset Y \supset E \arrow[dl,"\Theta\vert_Y",swap] \arrow[dr,"\varphi"] &  \\
   \mathbf{p_z} \in  Z\arrow[rr, "\pi", dashrightarrow] & &X \supset \Gamma
    \end{tikzcd}\]
    where $\Theta\vert_Y$ is the resolution of $\mathbf{p}_z$.

    By smoothness of $\Gamma$, the only singular point of $Z$ is $\mathbf{p}_z$. Hence $Z$ is terminal (resp. canonical) if and only if $\mathbf{p}_z$ is. By \cite[Theorem~Pg.~376]{YPG}, it follows that $Z$ is terminal (resp. canonical but not terminal) if and only if $n=k+2>r-1$ (resp. $n=k+2= r-1$). In other words, if and only if  $-K_{\Gamma}$ is nef (resp. $K_{\Gamma} \sim \mathcal{O}_{\Gamma}(1)$).

Suppose that $\alpha\vert_Y$ is not divisorial. Then, for the same reason as before, $\alpha\vert_Y$ is small, possibly an isomorphism. Suppose it is not an isomorphism. Then, the small modification $\alpha\vert_Y$ contracts $\PA \cap Y$ to a point. In particular $\PA \cap Y$ is non-empty which happens precisely when $g_{d-r} \not \in \mathbb C ^*$. In other words, when the following implication holds: If $r = d$ then $g_{d-r}$ is identically zero. Equivalently, $r=d$ implies $ \Pi \subset X$.
\end{proof}

\begin{lemma} \label{lem:rd}
    Let $\mathrm{deg} \Gamma = \mathrm{deg} X >1$. Then there is a projection $\pi \colon X \dashrightarrow \mathbb P^{n-k-1}$ that decomposes as
        \[
    \begin{tikzcd} Y \arrow[r,"\Theta\vert_Y",dashrightarrow]  \arrow[d,"\varphi",swap] &  Y' \arrow[d,"\psi"]\\  X  \arrow[r,dashrightarrow] &\mathbb P^{n-k-1}
    \end{tikzcd}
    \]
where $\Theta\vert_Y$ is a (non-isomorphism) small $\mathbb Q$-factorial modification if and only if $\Pi \subset X$ and an isomorphism otherwise.

In each case the general fibre of $\psi$ is a complete intersection $F$ such that $K_F \sim \mathcal{O}(d-k-3)$.
\end{lemma}

\begin{proof}
    Consider the flat projective family
    $ \eta \colon \mathcal{X} \rightarrow \mathbb C$ whose general fibre is the smooth degree $d$ hypersurface
    $$
    \mathcal{X}_{t} \colon (x_{k+1}F_{k+1}+\cdots+x_{n+1}F_{n+1}+t\cdot h_d=0) \subset \mathbb P^{n+1}
    $$
    containing $\Gamma \colon (h_d=0) \subset \Pi$, where $F_i$ and $h_d$ are as in \cref{set:hyp}. Notice that each fibre of $\eta$ is isomorphic to some $X$ as in the conditions of the lemma. Namely, the general fibre of $\eta$, $\mathcal{X}_t$, does not contain $\Pi$ while the central fibre, $\mathcal{X}_0$, does.
    In the following we study the birational geometry of the fibres of $\eta$.

   Let $\varphi_t \colon \mathcal{Y}_t \longrightarrow \mathcal{X}_t$ be the blowup of $\mathcal{X}_t$ along $\Gamma$ and similarly, let $\varphi_0\colon \mathcal{Y}_0 \longrightarrow \mathcal{X}_0$ be the blowup of $\mathcal{X}_0$ along~$\Gamma$. By Lemma \ref{lem:restr}, we know that $\Theta\vert_{\mathcal{Y}_t}\colon \mathcal{Y}_t \rightarrow \mathcal{Y}'_t$, the restriction of $\Theta$ to the general fibre $\mathcal{Y}_t$, is a small modification, possibly an isomorphism.
    Then, by Lemma \ref{lem:stricttransf}, $\mathcal{Y}_t$ is given by
    \begin{align*}
        x_{k+2}F^{'}_{k+2}+\cdots +x_{n+1}F^{'}_{n+1} + t\cdot z&=0 \\
        zu-h_d&=0
    \end{align*}
    inside~$T$. Note that $\mathcal Y_t$ is isomorphic to the smooth hypersurface
    $$
    u(x_{k+2}F^{'}_{k+2}+\cdots +x_{n+1}F^{'}_{n+1}) + t\cdot  h_d=0
    $$
    inside $\mathrm{Bl}_{\Pi} \mathbb P^{n+1}$. By Lemma \ref{lem:restr}\labelcref{item:iso},
    $\Theta\vert_{\mathcal{Y}_t}$ is an isomorphism. Let $\Psi$ be as in Lemma \ref{lem:toric}. By Lemma \ref{lem:toric}, it follows that $\psi_t:=\Psi\vert_{\mathcal{Y}'_{t}}$ is a fibration to $\mathbb P ^{n-k-1}$ whose general fibre is isomorphic to the smooth hypersurface  $F$
   $$
    u(F^{'}_{k+2}+\cdots +F^{'}_{n+1}) + t\cdot  h_d=0
    $$
    of degree $d$ in $\mathbb P ^{k+2}$. Hence, $K_F \sim \mathcal{O}_F(d-k-3)$. Hence, the general fibre of $\eta$ fits into the diagram

    \[
    \begin{tikzcd} \mathcal{Y}_t\arrow[r,equal]  \arrow[d,"\varphi_t",swap] &  \mathcal{Y}'_t \arrow[d,"\psi_t"]\\  \mathcal{X}_t\arrow[r,dashrightarrow] &\mathbb P^{n-k-1}
    \end{tikzcd}
    \]

On the other hand, the central fibre of the family $\eta$, $\mathcal{X}_0$, contains the $(k+1)$-dimensional linear subspace $ \Pi \colon ( x_{k+2}= \ldots = x_{n+1}=0) \supset \Gamma$ and so $\mathcal{Y}_0$ admits a small $\mathbb Q$-factorial modification $\Theta_0 \colon \mathcal{Y}_{0} \dashrightarrow \mathcal{Y}'_{0}$ by Lemma \ref{lem:restr}. By Lemma \ref{lem:toric}, it follows that $\Psi$ is a fibration to $\mathbb P^{n-k-1}$ whose general fibre is isomorphic to the complete intersection
 \begin{align*}
        F^{'}_{k+2}+\cdots +F^{'}_{n+1} &=0 \\
        zu-h_d&=0
    \end{align*}
of degrees $d-1$ and $d$, respectively, inside $\mathbb P (1^{k+3},d-1)$. Moreover it is clear that $\mathcal{Y}_0$ is indeed the central fibre of the family $\mathcal{Y}_t \rightarrow \mathbb C$. In particular, we have a degeneration of the corresponding movable cones where
\begin{align*}
\mathcal{Y}_t  & \rightsquigarrow \mathcal{Y}_0\\
\mathrm{Mov}(\mathcal{Y}_t)=\mathrm{Nef}(\mathcal{Y}_t) & \rightsquigarrow \mathrm{Mov}(\mathcal{Y}_0)=\mathrm{Nef}(\mathcal{Y}_0) \cup \mathrm{Nef}(\mathcal{Y}'_0)
\end{align*}
\end{proof}

\begin{example} \label{ex:jump}
In \cite{defI,defII}, Wi\'sniewski showed that the nef cone of a smooth Fano variety remains constant under deformations in smooth families of Fano varieties.
We give an example where the above fails if we allow the central fibre to be weak Fano. See also \cite{jump}.

Let $X_0 \in |\mathcal{O}_{\mathbb P^5}(2)|$ be a smooth quadric fourfold containing $\Pi\simeq \mathbb P^2$. Let $Y_0$ be the blowup of $X_0$ along a smooth conic in $\Pi$, that we denote by $\Gamma$. Then $Y_0$ is a \emph{weak} Fano variety and its movable cone is
$$
\mathrm{Mov}(Y_0)=\mathbb{R}_+[H]\oplus \mathbb{R}_+[H-E]=\mathrm{Nef}(Y_0) \cup \mathrm{Nef}(Y_0')
$$
where  $Y_0'$ is the only small $\mathbb Q$-factorial modification of $Y_0$. See Lemma \ref{lem:rd}. We now consider a flat deformation $\eta \colon X \rightarrow \mathbb C$ of $X_0$ where each fibre $X_t := \eta^{-1}(t),\, t\not = 0$, is a smooth quadric fourfold containing $\Gamma$ but \emph{not} $\Pi$. Let $Y_t$ be the blowup of $X_t$ along $\Gamma$. Then $Y_t$ is a smooth Fano variety which is a flat deformation of $Y_0$ and its Movable cone agrees with the Movable cone of $Y_0$, see Lemma \ref{lem:rd}. However it contains only one Mori chamber,
$$
\mathrm{Nef}(Y_t)=\mathrm{Mov}(Y_t)
$$
and so the Nef cone of $Y_0$ is a proper subcone of the Nef cone of $Y_t$, for $t \not = 0$.
\end{example}

\begin{lemma} \label{lem:rdiffd}
    Suppose $r>1$. Assume $\deg X \not = \deg \Gamma$ and $X \not = \mathbb P^n$ or $\mathrm{codim}_X\Gamma>2$. Let $\pi \colon X \dashrightarrow \mathbb P^{n-k-1}$ be the projection away from $\Pi$. Then $\pi$ can be decomposed as
       \[
    \begin{tikzcd} Y\arrow[r,"\Theta\vert_Y",dashrightarrow]  \arrow[d,"\varphi",swap] &  Y' \arrow[d,"\psi"]\\  X\arrow[r, "\pi",dashrightarrow] &\mathbb P^{n-k-1}
    \end{tikzcd}
    \]
   where $\psi$ is a fibration onto $\mathbb P^{n-k-1}$ whose general fibre $F$ is isomorphic to a smooth complete intersection in $\mathbb P(1^{k+3},r-1)$ such that $K_F\sim \mathcal{O}_F(d-k-3)$.
\end{lemma}

\begin{proof}

By Lemma \ref{lem:restr}, it follows that $\Theta\vert_Y \colon Y \dashrightarrow Y'$ is a small $\mathbb Q$-factorial modification and that the contracted locus is $\PA \cap Y$. On the other hand it is clear that the ideal of $Y$ is contained in the ideal of $E\cap \PB$. Hence, $\alpha'$ contracts $E\cap \PB \simeq \mathbb P^{\mathrm{codim_X \Gamma}-1}$.
Notice that $\dim E\cap \PB<n-1$ since $\dim\Gamma>0$.

Finally, we restrict the $\mathbb{P}(1^{k+3},r-1)$-bundle over $\mathbb P^{n-k-1}$ of Lemma \ref{lem:toric} to $Y$. This is a fibration whose general fibre $F$ is isomorphic to the smooth
(by generic smoothness, see \cite[Theorem~III.10.2 and Lemma~III.10.5]{Hartshorne})
codimension $2$ complete intersection
\begin{align*}
    F^{'}_{k+2}+\cdots +F^{'}_{n+1} + zg_{d-r}&=0 \\
    zu-h_r&=0
\end{align*}
of hypersurfaces of degrees $d-1$ and $r$, respectively, where the polynomials $F_{i}^{'}$ are as in \cref{lem:stricttransf} and where $x_{k+2}, \ldots, x_n$ are general complex numbers.
By adjunction $K_F\sim \mathcal{O}_F(r+d-1-(k+3)-(r-1)) = \mathcal{O}_F(d-k-3) $ and the first part of the claim follows.
\end{proof}

\begin{example}[Smooth Calabi-Yau fibrations] \label{exa:Smooth Calabi-Yau fibrations}
The following example gives a particular construction of Calabi-Yau fibrations of relative Picard rank $1$. Let $X \in |\mathcal{O}_{\mathbb P^{n+1}}(k+3)|$ be a smooth Fano hypersurface. Let $\Gamma \subset \Pi$ be a smooth hypersurface such that $K_{\Gamma} \sim \mathcal{O}_{\mathbb P^{k+1}}(1)$. Suppose $\Gamma \subset X$ and $\Pi \not \subset X$. Then, by Lemma \ref{lem:rd}, $Y$ the blowup of $X$ along $\Gamma$ admits a fibration to $\mathbb P^{n-k-1}$ such that the general fibre $F$ satisfies $K_F \sim 0$. For instance the blowup of a smooth quartic threefold hypersurface along a smooth planar quartic curve has the structure of a fibration into $K3$ surfaces over $\mathbb P^1$.
\end{example}

\begin{theorem} \label{thm:main}
    \MainCones{}
\end{theorem}

\begin{proof}
     The statement follows from Lemmas \ref{lem:r1},  \ref{lem:restr}, \ref{lem:rd}, \ref{lem:rdiffd} where the birational models have been explicitly computed.
\end{proof}

\begin{remark}
\begin{enumerate}[label=(\alph*), ref=\alph*]
 \item \Cref{thm:main} gives an alternative proof that $\Gamma$ is not a maximal centre when $X$ is a Fano variety of Fano index~$1$. See \cite{KollarRigidity10names} for a historical account of superrigidity of smooth Fano hypersurfaces of Fano index 1.
\item It is also possible to show that some of the varieties $Y$ in \cref{thm:main} are Mori dream spaces using arguments in \cite{Ito}.
\end{enumerate}
\end{remark}

\begin{example}
Let $X \in |\mathcal{O}_{\mathbb P^{5}}(d)|$ be a smooth 4-fold hypersurface of degree $d$ containing the plane $\Pi$.  Then,  the projection $X \dashrightarrow  \Pi'$ can be decomposed as the blowup of $X$ along a smooth conic in $\Pi$, followed by a small modification swapping two projective planes, followed by a fibration onto $\Pi$. See Lemmas \ref{lem:r1} \ref{lem:restr}, \ref{lem:rd}, \ref{lem:rdiffd}. The small modification is a flip if $X=\mathbb P^4$, a flop if $X$ is a quadric and an anti-flip otherwise, which follows from the description of the Nef cone in each instance. See Theorem \ref{thm:main}. In this case, smoothness of the blowup is preserved when we have a flop or flip but the anti-flip introduces hypersurface singularities along the flopped plane. See \cref{lem:restr}\ref{item:sqm}.  Moreover $-K_F \sim \mathcal{O}_F(4-d)$. Hence, $X$ is birational to a Fano fibration whenever $1 \leq d \leq 3$, a Calabi-Yau fibration if $d=4$ and a fibration into canonically polarised varieties when $d>4$.
       \end{example}

We classify for which $\Gamma$ and $X$ we have that the blowup of $X$ along $\Gamma$ is a smooth Fano variety.

\begin{theorem} \label{thm:mainII}
    \MainFano{}
\end{theorem}

\begin{proof}
    The statement follows from \cref{thm:main} and $-K_Y \sim (n+2-d)H-(n-k-1)E$.
\end{proof}

We illustrate how Theorem \ref{thm:mainII} can be made more explicit in particular cases.

\begin{corollary}
 \label{ex:curvefano}
    Let $\Gamma$ be a smooth projective planar curve and let $\Pi$ be the plane containing it. Let $\varphi \colon Y \rightarrow X$ be the blowup of $X$ along $\Gamma$. Then $Y$ is a Fano variety if and only if one of the following holds:
    \begin{itemize}
        \item $ \Gamma $ is a line and $X$ is a linear, quadric or cubic hypersurface of dimension $n\geq 3$;
        \item $\Gamma$ is a conic and  $X$ is a
        \begin{itemize}
            \item $3$ or $4$-dimensional hyperplane or
            \item quadric hypersurface of dimension $n\geq 3$ such that $\Pi \not \subset X$; or
        \end{itemize}
        \item $\Gamma$ is a cubic and $X$ is a
        \begin{itemize}
            \item $3$-dimensional hyperplane or
            \item a cubic hypersurface of dimension $n\geq 3$ such that $\Pi \not \subset X$.
        \end{itemize}

    \end{itemize}
\end{corollary}

\begin{proof}
    See the proof of Corollary \ref{ex:surfacefano}.
\end{proof}

\begin{corollary} \label{ex:surfacefano}
    Let $\Gamma$ be a smooth projective surface contained in $\Pi\simeq \mathbb P^3$. Let $\varphi \colon Y \rightarrow X$ be the blowup of $X$ along $\Gamma$. Then $Y$ is a Fano variety if and only if one of the following holds:
    \begin{itemize}
        \item $ \Gamma $ is a plane and $X$ is a hypersurface of degree at most four and dimension $n\geq 4$;
        \item $\Gamma$ is a quadric surface and  $X$ is a
        \begin{enumerate}
            \item hyperplane of dimension $4 \leq n\leq 6$ or
            \item quadric hypersurface of dimension  $n \geq 4$ such that $\Pi \not \subset X$ or
            \item $4$-dimensional cubic hypersurface;
        \end{enumerate}
        \item $\Gamma$ is a cubic surface and $X$ is a
        \begin{enumerate}
            \item hyperplane of dimension $4$ or
            \item a cubic hypersurface of dimension $n\geq 4$ such that $\Pi \not \subset X$; or
        \end{enumerate}
        \item $\Gamma$ is a quartic surface and $X$ is a
        \begin{enumerate}
            \item hyperplane of dimension $4$ or
            \item a quartic hypersurface of dimension $n\geq 4$ such that $\Pi \not \subset X$.
        \end{enumerate}
   \end{itemize}
\end{corollary}

\begin{proof}
By Theorem \ref{thm:main}, if $Y$ is a Fano variety then $-K_{\Gamma}$ is nef. Hence $\Gamma$ is a smooth surface which is a hypersurface of degree at most $4$. Suppose $r=1$ or $r=d$ and $\Pi \not \subset X$. If $r>1$, then $X$ is any smooth Fano hypersurface of degree $r$ and dimension $n\geq 4$ such that $\Pi \not \subset X$. If $r=1$, $X$ is any smooth hypersurface of degree $1\leq d\leq 4$ and dimension $n\geq 4$. Fix $r>1$. Suppose, on the other hand, that $r\not =d$ or $r=d$ but $\Pi\subset X$. Then, by Theorem \ref{thm:main}, the dimension of $X$ is bounded above for fixed $r$ and we have
\begin{equation} \label{eq:1}
4\leq n\leq \frac{1+3r-d}{r-1}.
\end{equation}
 It is easy to see that, for each $2\leq r\leq 4$, the solutions to the inequality \ref{eq:1} are the triples
 $$
 (r,d,n) \in \{(4,1,4), (3,1,4), (3,2,4), (2,1, 4 \leq n\leq 6), (2,2, 4 \leq n\leq 5), (2,3,4) \}.
 $$
 In this case, if $X$ is not a hyperplane, we have $r\geq d \implies \Pi \subset X$. But by Lemma \ref{lem:sing}  $3= \dim \Pi \leq n-3$. Hence $n\geq 6$. This excludes the case $(r,n,d) = (3,2,4)$, i.e., of a cubic surface in a quadric of dimension $4$, as well as the cases $(2,2,n)$ where $n \in \{4,5\}$. There are examples in all other cases and so the conclusion follows.
 \end{proof}

  \begin{corollary} \label{cor:logFano}
      Suppose we are in the same conditions of Theorem \ref{thm:mainII}. Let $\iota_X=n+2-d$ be the Fano index of $X$. Then, the anticanonical divisor $-K_Y$ is big if and only if $\mathrm{codim_X \Gamma}<\iota_X+1$. In particular, if $-K_X \sim_{\mathbb Q}\mathcal{O}_X(1)$, then the blowup of $X$ along $\Gamma$ is not a log Fano variety.
  \end{corollary}

  \begin{proof}
      By Theorem \ref{thm:main}, the anticanonical divisor of the blowup of $X$ along $\Gamma$ is big if and only if
      $$
      \frac{n-k-1}{n+2-d} = \frac{\mathrm{codim}_X\Gamma-1}{\iota_X}<1.
      $$
     In particular suppose $\iota_X =1$. Then, $\mathrm{codim_X \Gamma}< 2$. But, by assumption, $\mathrm{codim_X \Gamma} \geq 2$. Hence, $\iota_X \geq 2$.
  \end{proof}

The following theorem is a characterisation of when $\varphi \colon Y \rightarrow X$, the blowup of $X$ along $\Gamma$, initiates  a Sarkisov link in the case where $X=\mathbb P^n$. See \cite{HM13}.

\begin{theorem} \label{thm: Sarkisovlink}
    Suppose $X=\mathbb P^n$. Let $\varphi \colon Y \rightarrow X$ be the blowup of $X$ along $\Gamma$. Then $\varphi$ initiates a Sarkisov link if and only if
$-K_{\Gamma}$ is nef and whenever $\mathrm{codim}_{X}\Gamma>2$ we have $\dim \Gamma \geq 2$.
\end{theorem}

\begin{proof}
    When $\mathrm{codim}_{X}\Gamma=2$ this is explained in Lemma \ref{lem:restr}. If $\mathrm{codim}_{X}\Gamma>2$, then $Y$ admits a small $\mathbb Q$-factorial modification $\Theta\vert_Y \colon Y \dashrightarrow Y'$, where $Y'$ has quotient singularities of type $$
    \frac{1}{r-1}(\underbrace{1,\ldots,1}_{k+2},\underbrace{0,\ldots,0}_{n-k-1})
    $$
    along $E\cap \PB \simeq \mathbb P^{n-k-1}$. Then, $Y'$ is terminal if and only if each singular point is terminal and the codimension of the singular locus is at least three. The latter is $\mathrm{codim}_{Y'} E\cap \PB = k+1\geq 3$, i.e., $k\geq 2$. Hence $n\geq 4$. On the other hand, each such quotient singularity is terminal if and only if $k+2>r-1$, i.e., if and only if $-K_{\Gamma}$ is nef. Moreover $-K_{Y}$ is movable by the description of the movable cone in Theorem \ref{thm:main}. Hence $Y'/\mathbb P^{n-k-1}$ is a Mori fibre space.
\end{proof}

\begin{example}
 Let $X=\mathbb P^3$ and $\Gamma \subset X$ be a smooth planar curve. Then $-K_Y$ is big (See Corollary \ref{cor:logFano}). It is a weak Fano threefold if and only if $\Gamma$ is a curve of genus $g$ and degree $r$ such that $(g,r) \in \{(0,1),(0,2),(1,3),(3,4) \}$. (See for instance Theorem \ref{thm:main}). Moreover, $\varphi$ initiates a Sarkisov link in the first three cases only (See \cref{thm: Sarkisovlink}). This recovers the information in the first column of \cite[Table~1]{BlancLamyWeak}.
\end{example}

The paper \cite{Mat97} contains examples of toric flops from smooth 4-folds to singular 4-folds. We give examples of non-toric flops from smooth 5-folds to singular 5-folds.

\begin{example} \label{exa:flop from smooth to singular}
    The following is an example of a Sarkisov link featuring a flop between a smooth variety and a singular one. Let $X=\mathbb P^5$ and $\Gamma \colon (h_3=0)  \subset X$ a smooth del Pezzo surface of degree $3$. Let $\varphi \colon Y \rightarrow X$ be the blowup of $X$ along $\Gamma$. By Lemma \ref{lem:restr} and Lemma \ref{lem:rdiffd}, the projection $\pi \colon X \dashrightarrow \mathbb P^1$ can be resolved and fits into the diagram,
     \[
    \begin{tikzcd} Y\arrow[r,"\Theta\vert_Y",dashrightarrow]  \arrow[d,"\varphi",swap] &  Y' \arrow[d,"\psi"]\\  X\arrow[r, "\pi",dashrightarrow] &\mathbb P^{1}
    \end{tikzcd}
    \]
    where $\psi \colon Y' \rightarrow \mathbb P^1$ is a fibration onto $\mathbb P^1$ whose general fibre is isomorphic to the cubic fourfold hypersurface
    $$
    (zu-h_3=0) \subset \mathbb P(1^5,2)
    $$
 By Lemma \ref{lem:restr}, the map $\Theta\vert_Y$ contracts $\PA\simeq \mathbb P^3$ and extracts $E\cap \PB\simeq \mathbb P^1$. The variety $Y'$ is singular along the extracted~$\mathbb P^1$, locally analytically the singularity is of type $\frac{1}{2}(0, 1, 1, 1, 1)$ around every point of~$\mathbb P^1$. Moreover, $-K_Y = 2(3H-E)$, that is, $-K_Y$ is nef but not ample. Hence, $-K_Y$ does not intersect $\PA$ nor  $-K_{Y'}:={\Theta\vert_Y}_*(-K_Y)$ intersects $E\cap \PB$. In particular, if $C$ is any curve contracted by $\alpha$ then $K_Y\cdot C=0$ and similarly for $K_{Y'}$. We conclude that $\Theta\vert_{Y}$ is a flop.
\end{example}

\section{Divisorial Stability}

Let $\Pi \cong \mathbb P^{k+1}$ be a linear subspace of $\mathbb P^{n+1}$, where $1 \leq k \leq n-2$.
    Let $X\subset \mathbb P^{n+1}$ be a smooth Fano hypersurface and $\Gamma \subset \Pi$ a smooth hypersurface of $\Pi$ contained in~$X$. Let $\varphi \colon Y \rightarrow X$ be the
    blow up of $X$ along $\Gamma$. Suppose $Y$ is a Fano variety. In this section we consider divisorial stability of $Y$ in two cases:
 \begin{enumerate}
     \item $\Gamma$ is a line and
     \item $\Gamma =\mathbb P^k$, where $k\gg 1$.
 \end{enumerate}
 By Theorem \ref{thm:main}, we have $\mathrm{Mov(Y)}=\mathrm{Nef(Y)}=\mathbb{R}_+[H]+\mathbb{R}_+[H-E]$. By Corollary \ref{ex:curvefano}, $Y$ is a Fano variety if and only if $3-\deg X - \dim \Gamma > 0$ which in the first two cases amounts to $\deg X \in \{1,2,3 \}$. In each case $Y$ fits into the diagram
 \[
     \begin{tikzcd}
   & E \subset Y \arrow[dl,"\varphi",swap] \arrow[dr,"\theta"] &  \\
   \Gamma \subset  X\arrow[rr, "\pi", dashrightarrow] & &\mathbb P^{l}
    \end{tikzcd}
\]
by \cref{lem:r1} and \cref{lem:rd}, for some $l\geq 1$.  where $\varphi$ is the blowup of $\Gamma$, $E$ is the $\varphi$-exceptional divisor and $\theta$ is a $\mathbb P^{n-l}$-bundle. Let $H:=\varphi^*\mathcal{O}_X(1)$ be the pullback of a hyperplane section of $X$ not containing $\Gamma$. Let $u \in \mathbb R_{\geq 0}$. We have $-K_Y-uE \equiv (n+2-d)H-(n-k-1+u)E$ and by Theorem \ref{thm:main},
    $$
-K_Y-uE\,\, \text{is Nef} \iff -K_Y-uE\,\, \text{is Pseudo-effective} \iff u \in [0,3-d+k].
$$
If $\Gamma \simeq \mathbb P^k$ for $k\geq 1$, we have the following intersection numbers:
\begin{align*}
&H^n=d \\
&H^{n-i}\cdot E^i  =0, && i<\mathrm{codim}_{X}\Gamma=n-k \\
&H^{n-i}\cdot E^{i}=(-1)^{n-k}\bigg(\binom{i-1}{n-k} (d-1)- \binom{i-1}{n-k-1}\bigg), &&  2 \leq n-k\leq i \leq n.
    \end{align*}

 \begin{definition}
 Let $X$ be a Fano variety, and let $f : Y \rightarrow X $ be a projective birational morphism such that $Y$ is normal and let $E$ be a prime divisor on $Y$.
\[
A_{X}(E)=1+\mathrm{ord}_E\big(K_Y-f^*K_X\big), \quad S_X(E)=\frac{1}{(-K_X)^n}\int_0^{\infty} \mathrm{vol}(f^*(-K_X)-uE)du.
\]
We define the \textbf{beta invariant} as
\[
\beta(E)=A_{X}(E) -S_X(E).
\]
 \end{definition}

 The following is a weaker notion of K-stability:
 \begin{definition}[{\cite[Definition~1.1]{Kentovolume}}]
     The Fano variety $X$ is said to be divisorially stable (respectively, semistable) if $\beta (F) > 0$ (respectively, $\beta(F) \geq 0$) for every prime divisor $F$
on $X$. We say that $X$ is divisorially unstable if it is not divisorially semistable.
 \end{definition}
Note that divisorial stability is a weak form of K-stability since we only consider divisors on $X$, not on birational models of $X$.
In particular, K-stable varieties are divisorially stable.
The following result is a generalisation of \cite[Lemma~3.22]{CalabiBook}
\begin{theorem} \label{thm:stabilityI}
    \MainStability{}
\end{theorem}

\begin{proof}
    By Corollary $\ref{ex:curvefano}$ $Y$ is a Fano variety if and only if $d \in \{1,2,3\}$. By Lemma \ref{lem:r1}, we have the following diagram
    \[
     \begin{tikzcd}
   & E \subset Y \arrow[dl,"\varphi",swap] \arrow[dr,"\theta"] &  \\
   \Gamma \subset  X\arrow[rr, "\pi", dashrightarrow] & &\mathbb P^{n-1}
    \end{tikzcd}\]
    where $\varphi$ is the blowup of $\Gamma$, $E$ is the $\varphi$-exceptional divisor and $\theta$ is a $\mathbb P^1$-bundle. Let $H:=\varphi^*\mathcal{O}_X(1)$ be the pullback of a hyperplane section of $X$ not containing $\Gamma$. Let $u \in \mathbb R_{\geq 0}$. By Theorem \ref{thm:main}, we have
    $-K_Y-uE \sim (n+2-d)H-(n-2+u)E$ is nef if and only if $0 \leq u \leq 4-d$. Also, notice that if $u\geq 4-d$ the divisor  $-K_Y-uE$ is not big so its pseudo-effective threshold is $\tau(E)=\tau=4-d$.

    The following intersection numbers are easily computed:
    \begin{align*}
&H^n=d \\
&H^{n-1}\cdot E = \cdots =H^2\cdot E^{n-2}=0\\
&H\cdot E^{n-1}=(-1)^n\\
&E^n=(-1)^{n+1}(d-n).
    \end{align*}
   So we compute the expected vanishing order of $E$ to be
   {\footnotesize
   \begin{align*}
   S_Y(E)&=\frac{1}{(-K_Y)^n} \int_0^{\tau} \mathrm{vol}(-K_Y-uE)\, \mathrm{d}u \\
   &=\frac{1}{(-K_Y)^n} \int_0^{\tau}\bigg ( \sum_{i=0}^{n} {\binom{n}{i}} (-1)^i (n+2-d)^{n-i}(n-2+u)^i H^{n-i}E^i  \bigg)\, \mathrm{d}u  \\
   &=\frac{1}{(-K_Y)^n} \int_0^{\tau} \bigg((n+2-d)^nd-(n-2+u)^{n-1}n(n+2-d)+(n-2+u)^n(n-d))\bigg)\, \mathrm{d}u \\
   &= \frac{(n-2+\tau)^n d \tau-\big((n-2+\tau)^n-(n-2)^n\big)(n+2-d)+\big((n-2+\tau)^n-(n-2)^n\big)\frac{n-d}{n+1}}{(n-2+\tau)^nd-2(n-2)^{n-1}(2n-d)}.
   \end{align*} \par}
Hence, $Y$ is divisorially unstable along $E$ if and only if $S_Y(E) > 1$, that is, if and only if
\begin{multline*}
   (n+2-d)^nd(3-d)+\frac{n-d}{n+1}\big((n+2-d)^{n+1}-(n-2)^{n+1}-(n-2)^n (n+1)\big) \\- (n+2-d)  \big( (n+2-d)^n-(n-2)^n-(n-2)^{n-1}n\big) >0.
\end{multline*}
Rearranging, this is equivalent to
$$\bigg(\frac{n-2}{n+2-d}\bigg)^{n-1}> \frac{(d-1)(n(d-1)-2)(n+2-d)}{9n^2-(5d+4)n+4(d-1)}.
$$
Notice that the inequality is trivially true for $d=1$ since the left-hand-side is positive. Assume $d\in \{2,3\}$.
Define the two functions of $n$
$$
p_d(n)=\bigg(\frac{n-2}{n+2-d}\bigg)^{n-1}, \quad q_d(n)=\frac{(d-1)(n(d-1)-2)(n+2-d)}{9n^2-(5d+4)n+4(d-1)}.
$$
Then, $p_d(n)$ is a strictly increasing function of $n$ such that $\lim_{n\to \infty}p_d(n) = e^{-\tau}$. In particular
$$
p_d(3)=\bigg(\frac{1}{5-d} \bigg)^2  \leq p_d(n) < e^{-\tau}, \quad \forall\,\, n\geq 3.
$$ On the other hand $q_d(n)$ is strictly increasing if $d=2$ and strictly decreasing if $d=3$. Moreover, we have $\lim_{n\to \infty}q_d(n) = \big(\frac{d-1}{3}\big)^2$. Hence, $q_d(n)$ is bounded (above if $d=2$ and below if $d=3$) by this value. Hence,
$$
q_2(n)
< \sup_{m\geq3}q_2(m)
= \frac{1}{9}
= \min_{m \geq 3} p_2(m)
\leq p_2(n).
$$
On the other hand,
$$
q_3(n)
> \inf_{m\geq3}q_3(m)
= \frac{4}{9}>e^{-1}
= \sup_{m\geq 3} p_3(m)
> p_3(n).
$$
We conclude that $Y$ is divisorially unstable if $d \in \{1,2\}$ as we claimed.
\end{proof}

\begin{theorem} \label{thm:divisorial unstability 1000}
     Let $\Gamma \simeq \mathbb P^k \subset X$ be a $k$-dimensional linear subspace  in a smooth Fano hypersurface $X \subset \mathbb P^{n+1}$ of degree $d$. Let $\varphi \colon Y \rightarrow X$ be the blowup of $X$ along $\Gamma$ and suppose that $Y$ is a Fano variety. Suppose $k \geq 3\cdot \deg X$. If $n \leq 1000$, then $Y$ is divisorially unstable. In particular, $Y$ is not K-stable and does not admit a K\"ahler-Einstein metric.
\end{theorem}

\begin{proof}
 By \cref{thm:main}, $Y$ is a Fano variety if and only if $3-d+k>0$. Hence we have the following inequalities,
 $$
 1\leq d \leq k+2 \leq n.
 $$
 Therefore for each $n$ there is a finite number of possibilities for degree of $X$ and dimension of $\Gamma$ for which $Y$ is a Fano variety. By definition $Y$ is divisorially unstable along $E$ if
    $$
    \int_0^{3-d+k} \mathrm{vol}(f^*(-K_X)-uE)du-(-K_X)^n>0
    $$
    Using the earlier computation of the intersection numbers we get the inequality
   \begin{multline*}
    (n+2-d)^nd(2-d+k)+(n+2-d)^{n+1} \sum_{i=n-k}^{n}\frac{\binom{n}{i}(-1)^{n-k+i}\bigg(\binom{i-1}{n-k} (d-1)- \binom{i-1}{n-k-1}\bigg)}{i+1}\\
    - \sum_{i=n-k}^{n}\frac{\binom{n}{i}(n+2-d)^{n-i}(-1)^{n-k+i}(n-k-1)^{i+1}\bigg(\binom{i-1}{n-k} (d-1)- \binom{i-1}{n-k-1}\bigg)}{i+1}\\
    -\sum_{i=n-k}^{n}\binom{n}{i}(n+2-d)^{n-i}(-1)^{n-k+i}(n-k-1)^{i}\bigg(\binom{i-1}{n-k} (d-1)- \binom{i-1}{n-k-1}\bigg) > 0
   \end{multline*}
    Let $n \leq 1000$. Then using a computer, we verified that the inequality holds whenever $1\leq d \leq k+2$.
\end{proof}

\begin{remark}
    The bound $k\geq 3 \cdot \deg X$ is not sharp and leaves behind well known cases such as the blowup of a quadric threefold along a line (See \cite[Lemma~3.22]{CalabiBook}). A more accurate and meaningful bound should be possible to find.
\end{remark}

We finish with the following conjecture:

\begin{conjecture} \label{conj:divisorial unstability}
    Let $\Gamma \simeq \mathbb P^k \subset X$ be a $k$-dimensional linear subspace in a smooth Fano hypersurface $X \subset \mathbb P^{n+1}$ of degree $d$. Let $\varphi \colon Y \rightarrow X$ be the blowup of $X$ along $\Gamma$ and suppose that $Y$ is a Fano variety. Then there exists a constant $c(d)$ depending only on $d$ such that $Y$ is divisorially unstable if and only if $k > c(d)$.
\end{conjecture}

\newcommand{\etalchar}[1]{$^{#1}$}
\providecommand{\bysame}{\leavevmode\hbox to3em{\hrulefill}\thinspace}
\providecommand{\MR}{\relax\ifhmode\unskip\space\fi MR }
\providecommand{\MRhref}[2]{%
  \href{http://www.ams.org/mathscinet-getitem?mr=#1}{#2}
}
\providecommand{\href}[2]{#2}

\vspace{0.5\baselineskip}
\ShowAffiliations{\\[1\baselineskip]}%

\end{document}